\documentclass[a4paper,12pt]{article}

\usepackage{amsthm,amsmath}
\usepackage{color}
\usepackage{enumerate}
\usepackage{natbib}
\usepackage{tikz}
\usepackage{authblk}
\usepackage[margin=1.25in]{geometry}
\RequirePackage[colorlinks,citecolor=blue,urlcolor=blue]{hyperref}

\renewenvironment{proof}{{\bfseries Proof.}}{\\ \hfill \qed\\}

\newtheorem{proposition}{Proposition}
\newtheorem{lemma}{Lemma}
\newtheorem{theorem}{Theorem}
\def\diag{\mathop{\rm diag}\nolimits}

\newenvironment{Assumptions}%
    {\begin{enumerate}[(\assName \arabic{Assumption})]\stepcounter{Assumption}
    }%
    {\end{enumerate}}
\def\assumption{\stepcounter{Assumption}\item}
\newcounter{Assumption} 
\def\assName{{\bf\rm A}}

\DeclareMathOperator*{\argmin}{arg\,min}
\DeclareMathOperator*{\argmax}{arg\,max}

\title{On Bayesian robust regression with diverging number of predictors}
\author[1,2]{Daniel Nevo}
\author[1,3]{Ya'acov Ritov}
\affil[1]{Department of Statistics, The Hebrew University of Jerusalem}
\affil[2]{Departments of Biostatistics and Epidemiology\\
	
	Harvard T.H. Chan School of Public Health}

\affil[3]{Department of Statistics, University of Michigan}
\date{}
\begin{document}
	\maketitle
\begin{abstract}
This paper concerns the robust regression model when the number of predictors and the number of observations grow in a similar rate. Theory for  M-estimators in this regime has been recently developed by several authors \citep{el2013robust,bean2013optimal,donoho2013high}.
 Motivated by the inability of M-estimators to successfully estimate the Euclidean norm of the coefficient vector,  we consider a Bayesian framework for this model. We suggest a two-component mixture of normals prior for the coefficients and develop a Gibbs sampler procedure for sampling from relevant posterior distributions, while utilizing a scale mixture of normal representation for the error distribution . Unlike  M-estimators, the proposed Bayes estimator is consistent in the Euclidean norm sense. Simulation results demonstrate the superiority of the Bayes estimator  over traditional estimation methods.
\end{abstract}

\section{Introduction}
When fitting a linear regression model to data, estimators robust to outliers are often desired. One popular approach to achieve robustness to outliers is to use M-estimators under a penalty function other than the quadratic one. This methodology is often termed as ``robust regression''. Classical results for robust regression are that the M-estimator of the coefficients vector is consistent and normally distributed \citep[see][Chap. 7]{huber2011robust}.
These results were obtained for the case $p$, the number of predictors,  is fixed or
grows slowly with the number of observations, $n$.  The case where $p$ grows faster
than $n$ have been drawing  a lot of a attention.  In that scenario, a popular
approach is to consider penalization based estimation methods, e.g., the Lasso \citep{tibshirani1996regression}.

We consider a different scenario. Assume  that $p<n$, yet $p$ grows at the same rate as $n$. That is, $p/n\rightarrow \kappa$ for some positive constant $\kappa < 1$. This scenario was
first recognized as an interesting one by \cite{huber1973robust}. It is, however,
only with the emergence of ``big data'' that researchers have begun to investigate the robust regression model under this regime. Asymptotic distribution and variance calculations were recently developed  for M-estimators in this regime \citep{el2013robust,karoui2013asymptotic,donoho2013high}. Arguably the main result is that while the obtained estimator is normally distributed, its variance differs from Huber's classical results. \cite{bean2013optimal} have further shown that, unlike the classical $p\ll n$ scenario, the optimal M-estimator, in terms of efficiency, is not obtained by maximizing the log density of the
errors. One  more striking result is that for Double-Exponential errors, for $\kappa$ larger than approximately
$0.3$, least squares regression is superior to least absolute deviations regression.

We argue in the current paper that M-estimation might be the wrong approach to the robust regression model in the
$p/n\rightarrow \kappa,\; 0< \kappa < 1$ regime. Our main motivation is as follows.
When using M-estimators, the estimation error of a single coefficient is in the usual order of $n^{-1/2}$ \citep{el2013robust}. However,
the error accumulated over the coefficient vector does not vanish when $n\rightarrow\infty$. We will further argue that when signal and the noise are of the same asymptotic order, only a few of the true coefficient values can be larger (in their absolute value) than $n^{-1/2}$. Putting it together, the estimation error of small coefficients, which are the majority,  is larger than their actual value (in absolute value) when using M-estimators. Thus, apparent large effects can be actually microscopic.

Recognizing this characteristic of the problem, shrinkage methods may be a better fit in this robust regression model. Shrinkage can be achieved by using either the aforementioned regularization methods or by using  a Bayesian approach with appropriate scaling of the prior hyperparameters. In this paper we consider the latter option.

It is well known that shrinkage can be achieved using a Bayesian methodology. Most of the discussion is centered in the normal error model. The James-Stein estimator, \citep{james1961estimation}, is an empirical Bayes estimator \citep{efron1973stein}; Ridge regression estimator is identical to what we get if assuming the regression coefficients are iid with normal prior, and a maximum \emph{a posteriori}  (MAP) estimator is used; and if we replace the normal prior with Laplace distribution prior we get the Lasso. The Laplace prior for the coefficients is actively researched. An efficient Gibbs sampler was suggested by \cite{park2008bayesian}. However, the Lasso attracts criticism from a Bayesian point of view, since the full posterior distribution of the coefficients vector does not attain the same risk rate as the posterior mode \citep{castillo2015bayesian}. Another, more recent, prior proposal is the horseshoe prior \citep{carvalho2010horseshoe,carvalho2009handling}, which has some appealing properties, at least when the design matrix is the identity \citep{van2014horseshoe}.

  Our description above implies that the coefficients can be separated to two groups: small-value and large-value coefficients (in their absolute size).  This perspective aligns with existing Bayesian variable selection literature, where priors are assigned hierarchically. First, coefficients are separated into two groups and  then a prior distribution is determined according to the group assignment. Often, though not necessarily, one of the priors is the degenerate distribution at zero. A leading example for this framework is SVSS, Stochastic Search Variable Selection \citep{george1993variable,george1997approaches}.  An alternative approach is to have a prior distribution on the number of  non-zero coefficients, to choose these coefficients uniformly, and then to have a prior on the non-zero coefficients \citep{castillo2012needles},

  In this paper we suggest a full Bayesian model for the robust regression model when $p/n\rightarrow \kappa,\; 0< \kappa < 1$. We choose the prior distributions and hyperparameters such that our prior knowledge on the design is taken into account. We then utilize a scale mixture of normal representation of the error distribution to construct a reasonably fast Gibbs Sampler.

The rest of the paper is organized as follows. In Section \ref{Sec:Achilles}, we present notation and model assumptions, and  claim that M-estimation should not be used in this $p$-close-to-$n$ regime.
In Section \ref{Sec:BysModel}, we introduce an hierarchical Bayesian model and then, in Section \ref{Sec:Sampling}, we present a Gibbs sampler for parameter estimation. Detailed example is given in Section \ref{Sec:example} where we also present simulation results. Section \ref{Sec:Discuss} offers conclusion remarks. Proofs are given in Section \ref{Sec:proofs}.
\section{Achilles heel of M-estimators when $p/n\rightarrow\kappa\in(0,1)$}
\label{Sec:Achilles}
We start with notations. We use $\| \cdot\|,\| \cdot\|_1 $ and $\|\cdot\|_\infty$ for the Euclidean norm, the $\ell_1$ norm and the maximum norm of a vector, respectively.  Throughout the paper we consider the model
\begin{equation}
\label{Eq:Model}
Y^{(n)}=X^{(n)}\beta^{(n)}+\epsilon^{(n)},
\end{equation}
where  $\epsilon^{(n)}$ is a vector of i.i.d random variables with a known density function
$f_\epsilon(\cdot;\theta)$ characterized by $\theta$, an unknown parameter,  $X^{(n)}$ is a matrix of random
predictors,  $\beta^{(n)}$ is an unknown parameter vector  we wish to
estimate. To improve clarity, we henceforth omit the superscript indicating that all model components given in \eqref{Eq:Model} depends on $n$ ($\theta$ excluded).
 We denote $X_i^T$ for the $i^{th}$ row of $X$. $X$ and $\epsilon$ are
assumed to be independent. We denote $\beta^0$ for the true value of $\beta$. For a given penalty function $\rho$,
the M-estimator  of $\beta$, $\hat{\beta}^\rho$, defined as
\begin{equation}
\label{Eq:MestDef}
\hat{\beta}^\rho=\argmin_\beta\sum_{i=1}^{n}\rho(Y_i-X_i^T\beta).
\end{equation}
If $\rho$ is convex one could alternatively solve the equation
$$\sum_{i=1}^{n}X_i^T\psi(Y_i-X_i^T\beta)=0, \quad \psi:=\rho'.
$$
Huber's classical result \citeyearpar{huber1973robust}  is that if
$p^2/n\rightarrow 0$  then $\sqrt{n}(\hat{\beta}^\rho-\beta^0)$ is asymptotically
normal with a covariance matrix
$$
\frac{E(\psi^2)}{E^2(\psi')}\lim\limits_{n\rightarrow\infty}(X^TX)^{-1}
$$
If it is further assumed that $E(\epsilon_1)=0$, then by using general M-estimation theory it can be shown that this result holds for $p/n\rightarrow 0$.
\cite{portnoy1984asymptotic,portnoy1985asymptotic} derived consistency and asymptotic normality of M-estimators in the robust regression model under weaker assumptions. See also
\cite{maronna1981asymptotic}.

We now claim that in the model described above, a robust regression model where the number of predictors and the  number of observations are similar, M-estimators have undesirable properties. But first, we present our model assumptions:
\begin{Assumptions}
	\def\assName{{\bf\rm M}}
	\item $\lim\limits_{n\rightarrow \infty}\frac{p}{n}=\kappa \in (0,1)$.
	\label{Ass:kappa}
	\assumption The rows of $X$ are i.i.d $N(0,\Sigma)$ for a known covariance matrix sequence $\Sigma=\Sigma_p$. Furthermore, assume the eigenvalues of $\Sigma$ are bounded away from zero, for all $p$. \label{Ass:X}
	\assumption \label{Ass:epsiStandard} $\epsilon_i, i=1,2,...,n$ are i.i.d mean zero random variables with a
	density function $f_\epsilon$. $\ell_\epsilon=\log f_\epsilon$ is concave, bounded from above
	and has  three bounded derivatives, such that $\inf_{|t|<M}\ell''_\epsilon(t)>0$ for any $M<\infty$.
	\assumption \label{Ass:epsiFurther} $f_\epsilon$ is symmetric and for the function $g_\epsilon(u)=f_\epsilon(\sqrt{u})$ we have, for $u>0$ and $k=1,2,\dots$,
	$$
	\left(-\frac{d}{du}\right)^kg(u)\ge0.
	$$
\end{Assumptions}
 Assumption \eqref{Ass:epsiFurther} would be exploited when we will consider the Bayesian formulation of the problem. We note in passing that this assumption is fulfilled by rich family of distributions, such as Student's T distribution, the Laplace distribution, and the more inclusive exponential power family \citep{andrews1974scale,west1987scale}.

  We now take the frequentest point of view, which we will later, in the next section, replace in a Bayesian perspective.   When estimating a multidimensional parameter a loss function is needed to aggregate over the different components. A natural loss function is the $\ell_2$ of the estimation errors. The following proposition motivates our discussion.
 \begin{proposition}
 	\label{Prop:Motiv}
 	Let assumptions \eqref{Ass:kappa}--\eqref{Ass:epsiStandard} and the regularity assumptions needed for Result 1 in \cite{el2013robust} hold.  Let $\hat{\beta}^\rho$ be the M-estimator defined in \eqref{Eq:MestDef} with respect to a non-linear convex function $\rho$. Then $\|\hat{\beta}^\rho-\beta^0\|=O_p(1)$
 \end{proposition}
A proof using the results of \cite{el2013robust} is given in Section \ref{Sec:proofs}. Proposition \ref{Prop:Motiv} implies that the so-called $\ell_2-$consistency cannot be achieved by M-estimators under this regime.

However, we now claim M-estimators might be the wrong approach here. Consider specifically the statistical interesting problem arising when the signal and the noise are of the same asymptotic order, i.e, when $X_i^T\beta^0=O_p(1)$. For a moment, let $\Sigma=I$. Then, $X_i^T\beta^0=O_p(1)$ implies $\|\beta^0\|=O_p(1)$. If $\Sigma\ne I$, then since $\Sigma$ is known, the last statement holds if taking $\tilde{X}_i=X_i\Sigma^{-1/2}$ instead of $X_i$. Informally, having the number of predictors in the same scale as the number of observations while considering a finite signal-to-noise ratio, that does not vanish as $n$ grows, implies additional assumptions on $\beta^0$ structure. For example, not too many components of $\beta^0$ can be much larger than $n^{-1/2}$ (in their absolute value), otherwise the signal would be  stronger than assumed.   On the other hand, if $\beta^0$ is concentrated very close to zero for all $n$, then the signal is too week as $X_i^T\beta^0$ is too small.

M-estimation is invariant to translation of $\beta^0$, but some $\beta^0$ values are less expected then others, as we argued in the preceding paragraph. Therefore, another approach, which exploits that knowledge on the parameter vector $\beta^0$ is desirable. Since we know that many of the true coefficients  are smaller than $n^{-1/2}$ (in their absolute value) we could potentially gain better estimates if we  shrink some coefficients towards zero. This can be done using regularization based methods, or alternatively, using a Bayesian approach in a way that shrinkage is encouraged by the specified prior distribution. In this paper, we choose to take the latter option, and in the next section we develop such a Bayesian hierarchical model.

\section{A Bayesian Model}
\label{Sec:BysModel}
A Bayesian model for the robust regression involves at least one more level of
parameters. Assume model \eqref{Eq:Model} holds with a density function $f_\epsilon$ that depends on a parameter $\theta$, and obeys assumptions \eqref{Ass:epsiStandard}-\eqref{Ass:epsiFurther}.  We start with a prior for $\beta$.  As we argued in the previous section, we expect many of its components to be small, while some are considerably larger. To accommodate for this, we present a mixture prior for each $\beta_j, j=1,..,p$, with two normal mixture components. Denote $T=(t_1,...,t_p)\in \{1,2\}^p$ for the vector that indicates for each component $j$ if $\beta_j$ has a large variance $(t_j=1)$ or a small one $(t_j=2)$. We also denote $\phi=E(\sum\limits_{j=1}^{p}\mathbf{1}\{t_j=1\})$ for the number of ``large'' components in $T$. Finally, let $\delta^2_k, k=1,2$ be the variance in each of the mixture components.  Putting all together, the following assumptions depict our prior distributions.
\begin{Assumptions}
	\def\assName{\bf\rm P}
	\setcounter{Assumption} {1}
	\item The prior for $\beta$ is  of an iid mixture of two zero-mean normal variables
	$$
	(t_j-1)|\phi \sim Ber(\phi/p) \qquad \beta_j|t_j,\delta_1,\delta_2 \sim N(0,\delta^2_{t_j})
	$$
	\label{Ass:PriorForm}
	\assumption  $\phi=o_p(n/\log n)$.\label{Ass:PriorqRate}
	\assumption  $\delta^2_1$ is $O_p(\phi^{-1})$ and  $\delta^2_2$ is $O_p(1/n^\xi)$,
	for some fixed $\xi>1$. \label{Ass:deltasOrder}
\end{Assumptions}
Assumption \eqref{Ass:PriorqRate} implies that $\phi$ grows with $n$, but yet it is much smaller than $p$. Note that under assumptions
\eqref{Ass:PriorForm}--\eqref{Ass:deltasOrder} we have
\begin{equation*}
E(\|\beta\|^2 |\delta_1,\delta_2,\phi)=\phi\delta^2_1+(p-\phi)\delta^2_2=O_p(1),
\end{equation*}
so together with the Assumption \eqref{Ass:X} we get that $X_i^T\beta=O_p(1)$.

\textit{Working} prior distributions can be taken for $\phi, \delta_1$ and $\delta_2$, adding another level of hierarchy to the model. In practice, the parameters of these prior distributions are chosen such that  assumptions \eqref{Ass:PriorqRate} and \eqref{Ass:deltasOrder} are fulfilled. In Section \ref{Sec:example}, we present such an example.  Alternatively, $\phi, \delta_1$ and $\delta_2$ may be taken as known values that obeys Assumptions \eqref{Ass:PriorForm}--\eqref{Ass:deltasOrder}. The prior distribution for $\beta$  as specified in Assumptions
\eqref{Ass:PriorForm}--\eqref{Ass:deltasOrder} reflects our knowledge on $\beta$ when we assume the signal and the noise in our model are of the same order. This prior
implies that only a small part of the coefficients can be larger than the $n^{-1/2}$
threshold. This can be stated  formally using Chebyshev's inequality; See Lemma
\ref{Lem:Prr} in Section \ref{Sec:proofs}. Of course, one can think of other prior distributions for $\beta$ having this property.

As for $\theta$, we assume a prior distribution $q(\theta)$, where with some abuse of notation,  $q$ always denotes  a density, and the particular relevant density would be clear from its
argument. Unlike other model parameters, this prior does not change with $n$.


Before moving to the estimation procedure, we present a theoretical result showing that unlike M-estimators, a Bayes estimator for $\beta$ can achieve $\ell_2-$consistency.   
As stated before, M-estimators in our regime are consistent when
considering each coordinate of the vector separately, but not when considering the
parameter vector as a whole. The following theorem  shows that Bayesian
estimator in the discussed model is consistent (in the Euclidean norm sense) for the
parameter vector.
\begin{theorem}
	\label{thm:BetaConsistency}
	Consider the model \eqref{Eq:Model} and assume
	\eqref{Ass:kappa}--\eqref{Ass:epsiStandard} and
	\eqref{Ass:PriorForm}--\eqref{Ass:deltasOrder}. Let $\beta^0$ be the true value of
	$\beta$. Let $\hat{\beta}^*$ be a Bayes estimator with respect to the posterior distribution $q(\beta|Y,X)$ and a loss function $L$ of the form $L(\beta,\beta^0)=L(\|\beta-\beta^0\|)$, for a bounded $L$.  Then $\|\hat{\beta}^*-\beta^0\|\stackrel{p}{\to}0$ as $n\rightarrow \infty$,
\end{theorem}
where $\xrightarrow{p}$  denotes convergence in probability. The proof is given in Section \ref{Sec:proofs}.



\section{Sampling from the posterior distribution}
\label{Sec:Sampling}
Modern Bayesian statistics relies on the ability to sample from the posterior distribution. This may pose a challenge especially if the parameter space is high dimensional.  Assumption \eqref{Ass:epsiFurther} aids us in construction of a Gibbs sampler \cite{geman1984stochastic}, with blocked sampling available for the full conditional of $\beta$. Under the condition in Assumption  \eqref{Ass:epsiFurther}, we can write $f_\epsilon$ as a scale mixture of normal distribution, \cite{andrews1974scale},
\begin{equation}
\label{Eq:MixScale}
f_\epsilon(t;\theta)=\int_{0}^{\infty}\frac{1}{\sigma}\varphi\left(\frac{t}{\sigma}\right)q(\sigma^2|\theta)d\sigma^2
\end{equation}
where $\varphi$ is the density function of a standard normal random variable. This implies that while $q(Y|X,\beta,\theta)$ is $\prod_{i=1}^{n}f_\epsilon(Y_i-X_i^T\beta;\theta)$,  $q(Y|X,\beta,\sigma^2)$ is the density of independent normally distributed random variables with mean zero and variance $\sigma^2_i$. The mixing distribution $q(\sigma^2|\theta)$ can be identified in some cases \citep{andrews1974scale,west1987scale}. If $f_\epsilon$ is the density of a Laplace  distribution, as in the example we present in Section \ref{Sec:example}, then $q(\sigma^2|\theta)\propto \exp(-\theta^2\sigma^2/2)$. This representation adds $n$ parameters as an individual $\sigma^2_i$ is artificially introduced to the model for each observation $i$. However, it allows for direct sampling from all the full conditionals of all the parameters, and a Gibbs sampler can be used.  This Gibbs sampler resembles the one suggested for the Bayesian Lasso \citep{park2008bayesian}, especially for the case presented in our example in Section \ref{Sec:example} when the errors have a Laplace distribution. Note however that in their case the scale mixture of normals representation is taken for the prior $\beta$ and the errors are normally distributed, where here we apply this representation to the errors, and the prior for $\beta$ is a mixture of two normal distributions.

We now present the Gibbs sampler when $\Phi,\delta_1$ and $\delta_2$ are known hyperparameters. In Section \ref{Sec:example} we demonstrate how standard conjugate priors can be used for these parameters in a concrete example. The Gibbs sampler iterates between the full conditionals of $\theta,\sigma^2_1,...,\sigma^2_n,t_1,...,t_p$ and $\beta$. Starting with $\theta$, one can sample from $q(\theta|\sigma^2_1...,\sigma^2_n)\propto q(\theta)\prod_{i=1}^{n}q(\sigma^2_i|\theta)$. The mixing distribution $q(\sigma^2_i|\theta)$ is determined by the error distribution $f_\epsilon(t;\theta)$, and depending on the prior $q(\theta)$, we may have a conjugate family (as in our Section \ref{Sec:example} example). Alternatively, we may use Metropolis-Hastings step for $\theta$ only. Since $\theta$ is one dimensional, we expect such a step to marginally affect the computation time of the Bayes estimator.

Moving to each $\sigma^2_i$, we have
\begin{equation}
\label{Eq:sigmaFullConditional}
q(\sigma^2_i|Y_i,X_i, \beta,\theta)\propto \frac{1}{\sigma}\varphi\left(\frac{Y_i-X_i^T\beta}{\sigma_i}\right)q(\sigma^2_i|\theta)
\end{equation}
and depending on the mixture distribution this may be a known distribution  to sample from (see again Section \ref{Sec:example}).

Next, let $\Gamma=\diag(\sigma^2_1, \; \sigma^2_2, \; \dots, \; \sigma^2_n)$  and let $V=\diag(\delta^2_{t_1}, \; \delta^2_{t_2}, \; \dots, \; \delta^2_{t_p})$. The full conditional of $\beta$ is a multivariate normal  with mean $\mu$ and variance $A^{-1}$ where $A=X^T\Gamma^{-1}X+V^{-1}$ and $\mu=A^{-1}X^T\Gamma^{-1}Y$. This is the main advantage of the scale mixture of normals representation for $f_\epsilon$; block sampling of $\beta$ can be used. Finally, denote $p_k=P(t_j=k|\beta_j,\phi,\delta_1,\delta_2)$ for the full conditional of each $t_j$. We have
\begin{equation*}
(p_1\quad  p_2)\propto \left[\frac{1}{\delta_1}\varphi\left(\frac{\beta_j}{\delta_1}\right)\frac{\phi}{p} \quad \frac{1}{\delta_2}\varphi\left(\frac{\beta_j}{\delta_2}\right)\frac{p-\phi}{p}\right].
\end{equation*}
We now comment on alternative posterior sampling strategies, all involve the Gibbs sampler described above. When the mixing distribution $q(\sigma^2|\theta)$ results in a full conditional  $q(\sigma^2_i|Y_i,X_i, \beta,\theta)$  with no available direct sampling,  the proposed Gibbs sampler must use other MCMC procedures for sampling from  the full conditionals of $\sigma_i^2$. Moreover, if Assumption \eqref{Ass:epsiFurther} does not hold, then the scale mixture of normal representation \eqref{Eq:MixScale} cannot be used. Then, sampling from
\begin{equation*}
q(\theta|Y,X,\beta)\propto q(\theta)\prod_{i=1}^{n}f_\epsilon(Y_i-X_i\beta;\theta)
\end{equation*}
cannot necessarily be done in a direct sampling manner, but additional MCMC procedure, such as Metropoils-Hastings should be used. This is not, however, where most of the problem lies. Block sampling for $\beta$ would have been based upon high-dimensional full conditional of $\beta$
\begin{equation*}
q(\beta|Y,X,T,\delta_1,\delta_2)\propto \prod_{j=1}^{p}\delta^{-1}_{t_j}\varphi(\beta_j/\delta_{t_j})\prod_{i=1}^{n}f_\epsilon(Y_i-X_i\beta;\theta).
\end{equation*}
but since $p$ is large, this could be a too ambitious approach, when using Metropolis-Hastings or other MCMC procedures. So as a last resort, one could use a Gibbs sampler that uses $q(\beta_j|Y,X,T,\delta_1,\delta_2,\beta_{-j})$, where $\beta_{-j}$ is the $\beta$ vector without its $j$--th component.  Both the approaches we just described are expected to be more computationally heavy than the procedure we suggested before. We remark that sampling from each $q(\sigma^2_i|Y_i,X_i, \beta,\theta)$ can be parallelized, to shorten computation time  since the $\sigma^2_i$'s are independent given $\theta$ and $\beta$.

\section{Example}
\label{Sec:example}
We present in this section  a specific example.  First we describe the model with more detail, then we move to the estimation procedure and then we present simulation results with concrete numerical values. We assume a Laplace distribution for the errors, that is,
\begin{equation*}
\label{Eq:DBdensity}
f_\epsilon(t;\theta)=\frac{\theta}{2}\exp(-\theta|t|).
\end{equation*}
This implies that the mixing probability is Exponential, with $q(\sigma^2|\theta)\propto \exp(-\theta^2\sigma^2/2)$ \citep{andrews1974scale}. We take an Exponential prior distribution for $\theta^2$ (a Gamma prior distribution would also work). The resulting posterior for $\theta^2$ is a Gamma distribution with a shape parameter of $n/2+1$ and a scale parameter $1+(\sum\limits_{i=1}^{n}\sigma^2_i)/2$. Moving to $\delta_1$ and $\delta_2$. Let $N_k=\sum_{j=1}^{p}\mathbf{1}\{t_j=k\}$. Then given a prior distribution $q(\delta_k)$, the full conditional for $\delta_k$ is simplified to
\begin{equation*}
q(\delta_k|\beta,T)\propto (\delta_k^2)^{-\frac{N_k}{2}}\exp\left(\frac{1}{2\delta^2_k}\sum\limits_{\mathbf{1}\{t_j=k\}}\beta_j^2\right)q(\delta_k).
\end{equation*}
We take here a conjugate Inverse-Gamma prior distribution for $\delta_k$ with a shape parameter $\alpha_k$ and a scale parameter $\gamma_k$. The resulting posterior is Inverse-Gamma with parameters $\alpha'_k=\alpha_k+N_k/2$ and $\gamma'_k=\gamma_k+\sum\limits_{\mathbf{1}\{t_j=k\}}\beta_j^2/2$. An alternative prior distribution here would be the Jeffreys improper prior $q(\delta^2_k)\propto 1/\delta^2_k$. Next, a standard $Beta(\alpha_\phi,\gamma_\phi)$ prior is taken for $\phi/p$ so the resulting posterior is $Beta(\alpha_\phi+N_1,\gamma_\phi+N_2)$. We note here that the values taken for $\alpha_1,\alpha_2,\alpha_\phi, \gamma_1,\gamma_2$ and $\gamma_\phi$ should be on the background of Assumptions \eqref{Ass:PriorqRate}--\eqref{Ass:deltasOrder}. These are working priors that reflect our assumptions on $\beta$. The dimension of $\beta$ grows with $n$, while its $\ell_2$ norm is assumed to be constant. Therefore,  $\phi$, $\delta^2_1$ and $\delta^2_2$ must  change with $n$. This results in a working priors that also change with $n$.

Sampling from the full conditionals of $\beta$ and $T$ remains the same as described in Section \ref{Sec:BysModel}, and is generally the same regardless of $f_\epsilon$, as long as the scale mixture of normal distributions representation \eqref{Eq:MixScale} is used. The same property holds for the sampling from the full conditionals of $\delta_1,\delta_2$ and $\phi$ as described above.

Moving to $\sigma_i^2$, substituting $q(\sigma^2|\theta)\propto \exp(-\theta^2\sigma^2/2)$ in \eqref{Eq:sigmaFullConditional}, we have
\begin{equation*}
q(\sigma^2_i|Y_i,X_i, \beta,\theta)\propto \sigma^{-1}\exp\left(-\frac{(Y_i-X_i^T\beta)^2 + \theta^2\sigma_i^4}{2\sigma^2_i}\right).
\end{equation*}
Similarly to \cite{park2008bayesian}, this implies an Inverse-Gaussian distribution for $\sigma^{-2}_i$ with parameters $a = \theta^2$ and $b = \theta/|Y_i-X_i^T\beta|$, where the Inverse-Gaussian distribution defined as having the density
\begin{equation*}
f(t)=\sqrt{\frac{a}{2\pi}}t^{-3/2}\exp\left[-\frac{a(t-b)^2}{2b^2t}\right]\mathbf{1}\{t>0\}
\end{equation*}
\cite{chhikara1988inverse}. The full conditionals we just described are used in the proposed Gibbs sampler. We now turn to the presentation of simulation study results.
\subsection{Simulation Study}
\label{SubSec:Simulations}
We present simulation results for various $\kappa$ values and for $n=500,2500$. We used 500 and 1000 iterations for the Gibbs sampler described above. Data were simulated in the following way. First, $X$ was simulated under $\Sigma=I$, i.e., each of the iid rows of $X$ was simulated as $N_p(0,I)$. Then, we simulated $\phi, \delta_1$ and $\delta_2$. We used the values $\{\alpha_\phi=30, \beta_\phi=30(3\kappa\log(n)-1)\}, \{\alpha_1=2,\gamma_1=\log(n)/n\}$ and $\{\alpha_2=2, \gamma_2=n^{-1.5}\}$ for the prior distributions of $\phi/p, \delta^2_1$ and $\delta^2_2$, respectively. These values were chosen so $E(\phi)=n/(3\log(n)), E(\delta_1^2)=\log(n)/n,E(\delta_2^2)=n^{-1.5}$ and also to have prior  distributions that are not too concentrated around their means.  Given $\phi$, $T$ was simulated as iid $Ber(\phi/p)$. Then, we simulated $\beta$ as independent normally distributed variables where the prior variance of each $\beta_j$ is $\delta^2_{t_j}$.   Next, we simulate $\theta$ from an Exponential prior $q(\theta)= \exp(-\theta)\mathbf{1}\{\theta>0\}$, and simulate $\epsilon$ as iid random variables that follow Laplace distribution with parameter $\theta$. Finally, the observed data is $\{X,Y\}$, with $Y=X\beta+\epsilon$.

We compared the Bayesian estimator to classical M-estimators, namely least squares and least absolute deviations estimators. We also considered standard regularization-based estimators. Let the objective function to be minimized written as
\begin{equation*}
f(\beta)=\sum_{i=1}^{n}\rho(Y_i-X_i^T\beta)+P_\lambda(\beta)
\end{equation*}
 where $P_\lambda$ is a regularization function and $\lambda$ is a tuning constant chosen here by cross-validation. We considered the following four estimators,  defined by the choice of $\rho$ and $P_\lambda$: Least squares Lasso $(\rho(x)=x^2, P_\lambda(\beta)=\|\beta\|_1)$,  least squares Ridge regression $(\rho(x)=x^2, P_\lambda(\beta)=\|\beta\|^2)$, least absolute deviation Lasso $(\rho(x)=|x|, P_\lambda(\beta)=\|\beta\|_1)$ and least absolute deviation Ridge regression $(\rho(x)=|x|, P_\lambda(\beta)=\|\beta\|^2)$. The first two estimators were obtained using the algorithm described in \cite{friedman2010regularization} and the latter two estimators were calculated using the algorithm  described in \cite{yi2015semismooth} (\texttt{hqreg} package in \textbf{R}).

Table \ref{Tab:Results} presents medians of $\|\hat{\beta}-\beta\|^2$ across 1000 simulations for the estimators we just described. We report the medians and not the means due to a small number of outliers, observed mainly for the regularization based methods. For each method, the $\ell_2$ error grew with $\kappa$ (or with $p$) for a fixed $n$, as one may have expected. The superiority of the proposed Bayes estimator is clearly shown.  For $n=500$ taking 1000 iterations of the Gibbs sampler resulted only in minor gain comparing to using the Gibbs sampler with 500 iterations. For $n=2500$, taking $1000$ iterations was essential to improve this estimator's performance.

Considering Proposition \ref{Prop:Motiv},  the relatively poor performance of the non regularized M-estimators was expected. As expected, for $\kappa>0.3$ the least absolute deviation (LAD)  estimator was preferable over the least squares (LS) estimator. Among the regularization based methods, however, taking the least absolute deviation as a penalty function remained superior for all $\kappa$ values, both for Lasso and Ridge regression. Comparing the Lasso and Ridge regression, the former showed better performance for a fixed penalty function.
\begin{table}[ht]
	\small
	\centering
\caption{Medians values of $\|\hat{\beta}-\beta\|^2$ across 1000 simulations for the proposed Bayes estimator using 500 and 1000 iterations of the Gibbs sampler (Bayes$_{500}$ and Bayes$_{1000}$), the Lasso with square (LAS$_{LS}$) and absolute (LAS$_{LAD}$) deviations  penalty functions, Ridge regression with  square (RID$_{LS}$) and absolute (RID$_{LAD}$) deviations  penalty functions and the standard least squares (LS) and least absolute deviations (LAD) estimators }
\label{Tab:Results}
	\begin{tabular}{ccccccccc}
		\hline
		\multicolumn{9}{c}{$n=500$}\\
	$\kappa$&	 Bayes$_{500}$ & Bayes$_{1000}$  & LAS$_{LS}$  & LAS$_{LAD}$  & RID$_{LS}$ & RID$_{LAD}$ &  LS & LAD \\
\hline
 0.1  & 0.080 & 0.083 & 0.119 & 0.103 & 0.115 & 0.099 & 0.328 & 0.252 \\
 0.2  & 0.104 & 0.095 & 0.127 & 0.115 & 0.136 & 0.128 & 0.643 & 0.584 \\
 0.3  & 0.118 & 0.111 & 0.145 & 0.130 & 0.163 & 0.154 & 1.192 & 1.216 \\
 0.4  & 0.129 & 0.129 & 0.165 & 0.150 & 0.175 & 0.168 & 1.955 & 2.193 \\
 0.5  & 0.141 & 0.137 & 0.173 & 0.161 & 0.195 & 0.187 & 2.877 & 3.367 \\
 0.6  & 0.139 & 0.148 & 0.178 & 0.166 & 0.197 & 0.188 & 4.572 & 5.449 \\
 0.7  & 0.151 & 0.147 & 0.181 & 0.166 & 0.199 & 0.194 & 7.319 & 9.139 \\
 0.8  & 0.154 & 0.155 & 0.191 & 0.178 & 0.213 & 0.207 & 11.830 & 16.177 \\
 0.9  & 0.162 & 0.155 & 0.192 & 0.177 & 0.211 & 0.206 & 24.849 & 31.944 \\
 0.95 & 0.169 & 0.168 & 0.203 & 0.190 & 0.228 & 0.222 & 58.334 & 73.344 \\
		\hline		\multicolumn{9}{c}{$n=2500$}\\
			$\kappa$&	 Bayes$_{500}$ & Bayes$_{1000}$  & LAS$_{LS}$  & LAS$_{LAD}$  & RID$_{LS}$ & RID$_{LAD}$ &  LS & LAD \\
			\hline
 0.1  & 0.077 & 0.073 & 0.098 & 0.081 & 0.100 & 0.087 & 0.301 & 0.231 \\
 0.2  & 0.099 & 0.097 & 0.124 & 0.108 & 0.141 & 0.129 & 0.722 & 0.647 \\
 0.3  & 0.120 & 0.104 & 0.128 & 0.112 & 0.157 & 0.146 & 1.154 & 1.165 \\
 0.4  & 0.127 & 0.120 & 0.145 & 0.128 & 0.171 & 0.163 & 1.985 & 2.119 \\
 0.5  & 0.143 & 0.119 & 0.139 & 0.126 & 0.167 & 0.161 & 3.056 & 3.569 \\
 0.6  & 0.152 & 0.132 & 0.152 & 0.141 & 0.183 & 0.178 & 4.245 & 5.142 \\
 0.7  & 0.155 & 0.138 & 0.155 & 0.142 & 0.184 & 0.178 & 6.737 & 8.396 \\
 0.8  & 0.172 & 0.146 & 0.162 & 0.149 & 0.193 & 0.191 & 12.157 & 15.799 \\
 0.9  & 0.166 & 0.150 & 0.165 & 0.153 & 0.195 & 0.188 & 26.930 & 35.734 \\
 0.95 & 0.180 & 0.157 & 0.171 & 0.156 & 0.208 & 0.204 & 58.033 & 72.705 \\
\end{tabular}
\end{table}
Bayesian methods are often time consuming, especially when the target parameter is high dimensional and sampling from full conditionals is performed. The Gibbs sampler presented in this section involve direct sampling for all the parameters, with blocked sampling for $\beta$, without using additional MCMC steps.   Table \ref{Tab:Times} compares median computation time in minutes between  the different methods. Regularized least squares penalty function methods were considerably faster that the alternatives (excluding simple LS and LAD). However, as Table \ref{Tab:Results} suggests, they were also inferior to the other methods. The suggested Bayesian Gibbs sampler was comparable, and often preferable, in terms of computation time to the regularized least absolute deviations estimators.
\begin{table}[ht]
		\small
			\centering
	\caption{Medians computation times (minutes) across 1000 simulations for the proposed Bayes estimator using 500 and 1000 iterations of the Gibbs sampler (Bayes$_{500}$ and Bayes$_{1000}$), the Lasso with square (LAS$_{LS}$) and absolute  (LAS$_{LAD}$) deviations penalty functions, Ridge regression with  square (RID$_{LS}$) and absolute deviations (RID$_{LAD}$) penalty functions and the standard least squares (LS) and least absolute deviation estimators (LAD)}
	\label{Tab:Times}

	\begin{tabular}{cccccccc}
		\hline
		\multicolumn{8}{c}{$n=500$}\\
		$\kappa$&	Bayes$_{1000}$  & LAS$_{LS}$  & LAS$_{LAD}$  & RID$_{LS}$ & RID$_{LAD}$ &  LS & LAD \\
		\hline
  0.2 & 1.4 & 0.0 & 0.6 & 0.0 & 0.7 & 0.0 & 0.0 \\
  0.4 & 2.2 & 0.0 & 3.7 & 0.0 & 4.5 & 0.0 & 0.0 \\
  0.6 & 2.8 & 0.0 & 6.9 & 0.0 & 6.7 & 0.0 & 0.0 \\
  0.8 & 5.9 & 0.1 & 20.7 & 0.1 & 18.7 & 0.0 & 0.0 \\
  0.9 & 7.6 & 0.2 & 28.2 & 0.1 & 22.8 & 0.0 & 0.0 \\
  0.95 & 8.5 & 0.1 & 34.2 & 0.1 & 26.7 & 0.0 & 0.0 \\
				\hline		\multicolumn{8}{c}{$n=2500$}\\
				$\kappa$&	 Bayes$_{1000}$  & LAS$_{LS}$  & LAS$_{LAD}$  & RID$_{LS}$ & RID$_{LAD}$ &  LS & LAD \\
				\hline
  0.2 & 12.2 & 0.2 & 9.2 & 0.3 & 11.0 & 0.0 & 0.6 \\
  0.4 & 74.6 & 0.7 & 73.8 & 0.9 & 91.8 & 0.1 & 3.0 \\
  0.6 & 39.7 & 0.5 & 56.4 & 0.5 & 65.3 & 0.1 & 2.5 \\
  0.8 & 98.2 & 1.5 & 141.2 & 1.0 & 156.2 & 0.2 & 3.9 \\
  0.9 & 137.8 & 3.0 & 190.9 & 1.2 & 205.9 & 0.2 & 4.2 \\
  9.5 & 396.6 & 2.6 & 353.4 & 1.4 & 385.8 & 0.6 & 4.6 \\
\end{tabular}
\end{table}
\section{Discussion}
\label{Sec:Discuss}
This paper provided a Bayesian alternative to frequentist robust regression when the number of predictors and
the sample size are of the same order. Standard M-estimators are inconsistent when considering the error accumulated over the vector. If it is further assumed that signal and the noise are of the same asymptotic order, then shrinkage of many coefficients is desirable. We presented an hierarchical prior for model parameters and constructed a Bayes estimator suitable for the problem. A scale mixture of normal distribution representation for the errors' distribution allowed us to build an efficient Gibbs sampler with blocked sampling for the coefficients.    Theorem \ref{thm:BetaConsistency} shows that under appropriate conditions, the Bayes estimator in this problem is consistent in the $\ell_2$ sense. This property does not hold for M-estimators in this design. 

 The Bayesian estimator also outperforms regularization based methods. Although these methods, at least for the quadratic penalty function, can be computed much faster, they provide point estimate only, where the Gibbs sampler provides the entire posterior distribution that can be used to inference. This issue was also pointed out in \cite{park2008bayesian}.  It was previously shown that the asymptotic distribution of M-estimators in this design is nontrivial \citep{el2013robust,donoho2013high,karoui2013asymptotic}. In \cite{karoui2013asymptotic}, the distribution of the M-estimator was studied as the limit of a Ridge regularized estimators.  As standard asymptotic theory does not apply here,  the full distribution of the Bayes estimator in this regime remains as a future challenge.

\section{Proofs}
\label{Sec:proofs}
\subsection{Proof of Proposition \ref{Prop:Motiv}}
First, by Lemma 1 in \cite{el2013robust} we can write
\begin{equation*}
\hat{\beta}^\rho-\beta^0\overset{\mathcal{D}}{=}\|\hat{\beta}^{\rho,simp}\|\Sigma^{-1/2}u
\end{equation*}
where $\overset{\mathcal{D}}{=}$ denotes equality in distribution, and where $u$ is a $p$-length vector distributed uniformly on the sphere of radius one and where
\begin{equation*}
\hat{\beta}^{\rho,simp}=\argmin_\beta\sum_{i=1}^{n}\rho(\epsilon_i-\tilde{X}_i^T\beta).
\end{equation*}
with $\tilde{X}_i^T=\Sigma^{-1/2}X_i$ being a mean zero multivariate normal vector with the identity matrix as its variance. Their lemma also asserts that $\|\hat{\beta}^{\rho,simp}\|$ and $u$ are independent. By Result 1 in \cite{el2013robust},  $\|\hat{\beta}^{\rho,simp}\|$ has a deterministic limit denoted here (and there) by $r_\rho(\kappa)$. They further show how $r_\rho(\kappa)$ can be found. Next, by a multivariate central limit theorem, for large enough $p$,  $\sqrt{p}\Sigma^{-1/2}u$ is approximately $N_p(0,\Sigma^{-1})$. Thus, for large enough $n$ we have
\begin{equation}
\label{Eq:normEq}
\|\sqrt{n}(\hat{\beta}^\rho-\beta^0)\|\overset{\mathcal{D}}{=}r^2_\rho(\kappa)\kappa^{-1}\|x\|^2 + o_p(1)
\end{equation}
where $x\sim N_p(0,\Sigma^{-1})$. Let $\Sigma=T\Lambda T^T$ be the spectral decomposition of $\Sigma$. For $v\sim N_p(0,I)$ and $v\overset{\mathcal{D}}{=}v^*$, we have
\begin{equation*}
\|x\|^2=v^TT\Lambda^{-1}T^Tv=(v^*)^T\Lambda^{-1}v^*\ge^{st.}\lambda_{min}^{-1}\chi^2_p
\end{equation*}
where $\ge^{st.}$ symbolizes larger in the stochastic sense,   $\lambda_{min}$ is the minimal  eigenvalue of $\Sigma$, and $\chi^2_p$ is a Chi-squared random variable with $p$ degrees of freedom. Therefore, the right hand side of \eqref{Eq:normEq} is $O_p(n)$ and we are done.
\subsection{Lemma \ref{Lem:Prr}: Presentation and proof}
Denote $B^{C,\eta}_n:=\frac{1}{p}\sum_j
\mathbf{1}\left\{|\beta_j|>Cn^{-\eta/2}\right\}$ for the proportion of coordinates of
$\beta$ that are of order larger than $Cn^{-\eta/2}$.  The following lemma ensures us
that if the prior distribution of $\beta$ admits Assumptions
\eqref{Ass:PriorForm}-\eqref{Ass:deltasOrder} then $B^{C,\eta}_n$ is not far from $\phi/p$,
and consequently, $B^{C,\eta}_n=O_p(\phi/p)=o_p(log(n)^{-1})$. Thus, the proportion of ``large'' coefficient values out of the total number of coefficients goes to zero with probability that goes to one as $n$ grows.
\begin{lemma}
	\label{Lem:Prr}
	Let assumptions \eqref{Ass:kappa}, \eqref{Ass:PriorForm} and \eqref{Ass:PriorqRate}
	hold. Assume \eqref{Ass:deltasOrder} holds with some $\xi$. Then, for all
	$1\le\eta<\xi$, for all $\zeta>0$ and for any constant $C$ we have
	\begin{equation*}
	\label{Eq:LemmaPrior}
	\lim\limits_{n\rightarrow \infty}P\Bigl(\Bigl|B^{C,\eta}_n-\frac{\phi}{p}\Bigr|>\frac{\phi}{p}\zeta\Bigr)=0
	\end{equation*}
\end{lemma}
\begin{proof}
	First note that $B^{C,\eta}_n$ is a mean of $p$ independent Bernoulli random
	variables with success probability of
	\begin{align*}
	\nu^{C,\eta}_n&=\frac{\phi}{p}2\biggl[\Phi\Bigl(-\frac{C}{\sqrt{n^\eta\delta_1^2}}\Bigr) + \Phi\Bigl(-\frac{C}{\sqrt{n^\eta\delta_2^2}}\Bigr)\biggr] + 2\Phi\Bigl(-\frac{C}{\sqrt{n^\eta\delta_2^2}}\Bigr)
	\\
	&= \frac{\phi}{p}2\Phi\Bigl(-\frac{C}{\sqrt{n^\eta\delta_1^2}}\Bigr) + o\Bigl(\frac{\phi}{p}\Bigr).
	\end{align*}
	with $\Phi$ being the CDF of a standard normal random variable. The second equality
	results from  $n^\eta\delta_2^2=n^{\eta-\xi}$ (Assumption
	\eqref{Ass:deltasOrder}) and since $\eta<\xi$. Now, this problem is symmetric, so it
	is suffice to show
	\begin{equation*}
	\label{Eq:LemmaPriorOneSide}
	\lim\limits_{n\rightarrow \infty}P\Bigl(B^{C,\eta}_n>\frac{\phi}{p}(1+\zeta)\Bigr)=0.
	\end{equation*}
	By Chebyshev's inequality we have
	\begin{align*}
	P\Bigl(B^{C,\eta}_n>\frac{\phi}{p}(1+\zeta))&\le
	P\Bigl(|B^{C,\eta}_n-\nu^{C,\eta}_n|>\frac{\phi}{p}(1+\zeta)-\nu^{C,\eta}_n\Bigr)\\&\le
	\frac{\nu^{C,\eta}_n}{p(\frac{\phi}{p}(1+\zeta)-\nu^{C,\eta}_n)^2}\\
	&=\frac{\frac{\phi}{p}2\Phi\Bigl(-\frac{C}{\sqrt{n^\eta\delta_1^2}}\Bigr)+o\Bigl(\frac{\phi}{p}\Bigr)}{\frac{\phi^2}{p}\Bigl(1+\zeta-2\Phi\Bigl(-\frac{C}{\sqrt{n^\eta\delta_1^2}}\Bigr)\Bigr)^2+o\Bigl(\frac{\phi^2}{p}\Bigr)}
	\end{align*}
	and the last expression goes to zero as $n\rightarrow\infty$.
\end{proof}

\subsection{Proof of Theorem \ref{thm:BetaConsistency}}
For the simplicity of the proof we will assume that $\Sigma=I$.

Let $\hat\beta^{MAP}=\argmax_{\beta}q(\beta|Y,X)$ be the MAP estimator.  First we will show that $\|\hat\beta^{MAP}-\beta^0\|=o_p(1)$, and then that $\|\hat\beta^{MAP}-\hat\beta^\star\|=o_p(1)$. We 

Starting with $\hat\beta^{MAP}$, will first show the weaker result $\|\hat\beta^{MAP}-\beta^0\|=O_p(1)$.
We will then build on this result to strengthen the conclusion. For any vector $\beta$,  we partition  $\beta$ to two subvectors $\beta_{{\cal M}}$ and $\beta_{{\cal M}^c}$, were {\cal M} is the subset of relatively large coefficients.  We will then use the fact that $\|\hat\beta_{{\cal M}}^{MAP}-\beta_{{\cal M}}^0\|=O_p(1)$ to argue that $\|\hat\beta_{{\cal M}^c}^{MAP}-\beta_{{\cal M}^c}^0\|=o_p(1)$, and the latter to argue that $\|\hat\beta_{{\cal M}}^{MAP}-\beta_{{\cal M}}^0\|=o_p(1)$. Now, for  the details.

By Taylor expansion, the log-posterior of $\beta$ can be written as
\begin{align}
&\nonumber\hspace{-2em}\log q(\beta|Y,X)
\\\nonumber&=  \sum\limits_{i=1}^{n}\ell_\epsilon(Y_i-X_i^T\beta)
+\sum_{j=1}^{p}\log(q\bigl(\beta_j)\bigr)
\\\label{Eq:post0}&=
 \sum\limits_{i=1}^{n} \ell_\epsilon \left(\epsilon_i-X_{i}^T(\beta-\beta^0)\right) + \sum_{j=1}^{p}\log(q\bigl(\beta_j)\bigr)
\\
\begin{split}\label{Eq:post}
&=
\sum\limits_{i=1}^{n}\ell_\epsilon(\epsilon_i) + \sum_{j=1}^{p}\log(q\bigl(\beta_j)\bigr) - (\beta-\beta^0)^T\sum\limits_{i=1}^{n}X_{i} \ell_\epsilon'\left(\epsilon_i\right)
\\ &\hspace{1em}+
\frac{1}{2} (\beta-\beta^0)^T \sum\limits_{i=1}^{n}\ell_\epsilon''\left(\epsilon_i+\alpha_\beta X_i^T(\beta-\beta^0))\right)X_{i}X_{i}^T(\beta-\beta^0),
\end{split}
\end{align}
for some $\alpha_\beta\in[0,1]$.

Since the empirical distribution function converges to the cumulative distribution function
there are  $M<\infty$ and $\gamma<1$ such that
\begin{equation}
\label{epsBounded}
P(\sum \mathbf{1}(|\epsilon_i|>M)  <\gamma n) \to 1.
\end{equation}

 Next we want to argue that we also have
 \begin{equation}
 \label{epsStarBounded}
 P\bigl(\sum \mathbf{1}(|\epsilon_i-X_{i}^T(\hat{\beta}^{MAP}-\beta^0)|>M)<\gamma n\bigr) \to 1.
  \end{equation}
Let $U_n\approx V_n$ if $U_n=O_p(V_n)$ and $V_n=O_p(U_n)$. Now
 \begin{align}
 \begin{split}
    E|\ell(\epsilon)-\ell(0)| &= 2\int_0^\infty \bigl|\ell(x)-\ell(0)\bigr|e^{\ell(x)}dx
    \\
    &=2\int_0^\infty x\bigl|\ell'(\alpha_x x)\bigr|e^{\ell(x)}dx
    \\
    &\le 2\int_0^\infty x\bigl|\ell'( x)\bigr|e^{\ell(x)}dx
    \\
    &= 2\int_0^\infty e^{\ell(x)}dx =2.
    \end{split}
  \end{align}
Hence $n^{-1}\sum_{i=1}^n \ell(\epsilon_i)\stackrel{p}{\to} E\ell(\epsilon)>-\infty.$  Thus, $\sum \ell(\epsilon_i)\approx n$. Since $\hat\beta^{MAP}$ is the maximizer we have that $ q(\hat\beta^{MAP}|Y,X)- q(\beta^0|Y,X)>0$. Consider this difference and recall also line \eqref{Eq:post0}. Consider first the difference in the prior contributions. If $t_j=2$ (see Assumption \eqref{Ass:PriorForm}) then the contribution of the prior can be made larger by making $\beta_j$ smaller, but then $|\beta_j-\beta_j^0|=O_p(n^{-\xi/2})$. If $t_j=1$ we can increase the prior contribution, but there are only $O_p(n/\log n)$ such terms. Thus, $|\sum_{j=1}^{p}\log(q\bigl(\beta_j^{MAP})\bigr) -\sum_{j=1}^{p}\log(q\bigl(\beta_j^0)\bigr)|=O_p(n/\log n)$.
Since $\hat\beta^{MAP}$ is improving over $\beta^0$ and $\sum_i^n\ell(\epsilon_i)\approx n$ we must have that for some $A>0$, $\sum \ell(\epsilon_i-X_{i}^T(\hat{\beta}^{MAP}-\beta^0))>-An$ with high probability. Since $\lim_{x\to\infty}=-\infty$, we can choose a finite $M$  such that both \eqref{epsBounded} and  \eqref{epsStarBounded} hold.  Let
\begin{equation*}
{\cal A}= \{i:\; |\epsilon_i|<M \text{ and } |\epsilon_i-X_{i}^T(\hat{\beta}^{MAP}-\beta^0)|<M \}.
\end{equation*}
We conclude from \eqref{epsBounded} and \eqref{epsStarBounded} that for  large enough $M$, $|\cal A|$ is the same order as $n$. Hence
\begin{align}
\label{Eq:Apsd}
\sum_{i=1}^{n} \ell''_\epsilon \bigl(\epsilon_i+\alpha_{\hat\beta^{MAP}}X_i^T(\hat\beta^{MAP}-\beta^0))\bigr)X_iX_i^T \le \ell''_\epsilon(M)\sum_{\cal A}X_iX_i^T
\end{align}
in the partial order of positive semi-definite matrices.  We conclude from the prior, \eqref{Eq:post}, and we can verify by \eqref{Eq:Apsd}  that  \begin{equation}
    \label{eq:db0}
    \|\hat\beta^{MAP}-\beta^0\|=O_p(\sqrt{p/n})=O_p(1).
\end{equation}
We now build on this result to strengthen the conclusion.
Let ${\cal M}\subseteq\{1,\dots,p\}$ be the set of indices such that $|\beta_j^0|> n^{-\xi}\log n$. Denote by $\beta^0_{\cal M}$ and $\hat\beta^{MAP}_{\cal M}$ the subvectors with indices in ${\cal M}$ of $\beta^0$ and $\hat\beta^{MAP}$, respectively.
Let ${\cal M}^c$ the complementary set and define the corresponding subvectors and submatrix similarly.

We start with estimation error of $\hat\beta^{MAP}_{{\cal M}^c}$. Let $R=\{\beta:\|\beta_{{\cal M}^c}-\beta^0_{{\cal M}^c}\|=n^{-\gamma} \text{ and } \beta_{\cal M}=\hat\beta^{MAP}_{\cal M}\}$, where $0<\gamma<\xi-1$ and consider
\begin{equation}
\label{Eq:Wdef}
W=\max_{\beta\in R}\log q(\beta|Y,X)-\log q(\beta^{0,MAP}|Y,X),
\end{equation}
where $\beta^{0,MAP}$ equals to $\beta^0$ on ${\cal M}^c$ and to $\hat\beta^{MAP}$ on ${\cal M}$.
Substituting \eqref{Eq:post} in \eqref{Eq:Wdef}, the second term on the RHS of \eqref{Eq:post} contributes to $W$  $-n^{\xi}n^{-2\gamma}$, the third term is at most $O_p(n^{-\gamma+1})$, while the last term is  by \eqref{Eq:Apsd} and \eqref{eq:db0}  at most $O_p(n^{-2\gamma+1}) +O_p(n^{-\gamma+1})$. Thus $W<0$, and since $\hat\beta^{MAP}_{{\cal M}^c}$ is the maximizer, $\|\hat{\beta}^{MAP}_{{\cal M}^c}-\beta^0_{{\cal M}^c}\|=o_p(1)$.

This means that the error in estimation $\beta_{{\cal M}^c}$ has negligible contribution. The estimating equation for $\beta_{{\cal M}}$ with only the $t=1$ part of the prior is a ridge regression with $|{\cal M}|=o_p(n)$ variables with negligible ridge penalty, thus standard   mean calculations and Assumption \eqref{Ass:kappa} would show that $\|\hat\beta_{\cal M}^{MAP}-\beta_{\cal M}^0\|=o_p(1)$. By Assumptions \eqref{Ass:PriorForm}--\eqref{Ass:deltasOrder} the $t=2$ part has a relative peak at $0$ of height of the ratio of the two standard deviations (Assumption \eqref{Ass:deltasOrder}) and the ratio of the mixing probabilities (Assumption \eqref{Ass:PriorForm}). It  contributes the log of the product of these two terms, $O(\log n)$, but only for components which are  in an $O(n^{-\xi/2})$ of 0. But the curvature of the likelihood is $O_p(n)$  and the corresponding components of $\beta^0$ are $O_p\bigl(\phi^{-1/2}\bigr)$, thus the gain in the prior cannot balance the loss in the likelihood which is of order $n\phi^{-1}$ per component which is this neighborhood of 0.

\color{red} 


\color{black}

Having established that the MAP estimator is in $o(1)$ neighborhood of the truth, we consider now $\hat\beta^*$.  Again, consider the partition to $\beta_{{\cal M}}$ and $\beta_{{\cal M}^c}$. The log-prior of any  $\beta_j$ is
\begin{equation*}
\begin{split}
\log\Bigl(\frac{c_1^{1/2}\phi^{3/2}}{p}e^{-c_1\phi\beta_j^2} + c_2^{1/2}n^{\xi/2}e^{-c_2n^{\xi}\beta_j^2} \Bigr),
\end{split}
\end{equation*}
where $c_1,c_2=O(1)$. Thus, for $|\hat\beta_j^{MAP}|<L_n n^{-\xi/2}$, for $L_n\to\infty $ slowly, the central component of the prior dominates, and the \emph{a posteriori} mass of this ball is negligible. For $\hat\beta_j^{MAP}\approx \phi^{-1/2}$, we argued that the peak at 0 is much smaller than the value at $\hat\beta^{MAP}_j$. Since the radius of  this peak is of order $n^{-\xi}$. It can be ignored. Ignoring this neighborhood of 0, the \emph{a posteriori} is strictly concave with maximum at $\hat\beta_j^{MAP}$, and curvature of order $n$ and  hence $\hat\beta^*_j-\hat\beta^{MAP}_j=O_p(n^{-1/2})$.

\qed

\bibliography{robust}

\begin{thebibliography}{26}
\providecommand{\natexlab}[1]{#1}
\providecommand{\url}[1]{\texttt{#1}}
\expandafter\ifx\csname urlstyle\endcsname\relax
  \providecommand{\doi}[1]{doi: #1}\else
  \providecommand{\doi}{doi: \begingroup \urlstyle{rm}\Url}\fi

\bibitem[Andrews and Mallows(1974)]{andrews1974scale}
David~F Andrews and Colin~L Mallows.
\newblock Scale mixtures of normal distributions.
\newblock \emph{Journal of the Royal Statistical Society. Series B
  (Methodological)}, pages 99--102, 1974.

\bibitem[Bean et~al.(2013)Bean, Bickel, El~Karoui, and Yu]{bean2013optimal}
Derek Bean, Peter~J Bickel, Noureddine El~Karoui, and Bin Yu.
\newblock {Optimal M-estimation in high-dimensional regression}.
\newblock \emph{Proceedings of the National Academy of Sciences}, 110\penalty0
  (36):\penalty0 14563--14568, 2013.

\bibitem[Carvalho et~al.(2009)Carvalho, Polson, and
  Scott]{carvalho2009handling}
Carlos~M Carvalho, Nicholas~G Polson, and James~G Scott.
\newblock Handling sparsity via the horseshoe.
\newblock In \emph{International Conference on Artificial Intelligence and
  Statistics}, pages 73--80, 2009.

\bibitem[Carvalho et~al.(2010)Carvalho, Polson, and
  Scott]{carvalho2010horseshoe}
Carlos~M Carvalho, Nicholas~G Polson, and James~G Scott.
\newblock The horseshoe estimator for sparse signals.
\newblock \emph{Biometrika}, page asq017, 2010.

\bibitem[Castillo et~al.(2012)Castillo, van~der Vaart,
  et~al.]{castillo2012needles}
Isma{\"e}l Castillo, Aad van~der Vaart, et~al.
\newblock Needles and straw in a haystack: Posterior concentration for possibly
  sparse sequences.
\newblock \emph{The Annals of Statistics}, 40\penalty0 (4):\penalty0
  2069--2101, 2012.

\bibitem[Castillo et~al.(2015)Castillo, Schmidt-Hieber, and Van~der
  Vaart]{castillo2015bayesian}
Ismael Castillo, Johannes Schmidt-Hieber, and Aad Van~der Vaart.
\newblock Bayesian linear regression with sparse priors.
\newblock \emph{The Annals of Statistics}, 43\penalty0 (5):\penalty0
  1986--2018, 2015.

\bibitem[Chhikara(1988)]{chhikara1988inverse}
Raj Chhikara.
\newblock \emph{The Inverse Gaussian Distribution: Theory: Methodology, and
  Applications}, volume~95.
\newblock CRC Press, 1988.

\bibitem[Donoho and Montanari(2013)]{donoho2013high}
David Donoho and Andrea Montanari.
\newblock {High Dimensional Robust M-Estimation: Asymptotic Variance via
  Approximate Message Passing}.
\newblock \emph{arXiv preprint arXiv:1310.7320}, 2013.

\bibitem[Efron and Morris(1973)]{efron1973stein}
Bradley Efron and Carl Morris.
\newblock Stein's estimation rule and its competitorsan empirical bayes
  approach.
\newblock \emph{Journal of the American Statistical Association}, 68\penalty0
  (341):\penalty0 117--130, 1973.

\bibitem[El~Karoui(2013)]{karoui2013asymptotic}
Noureddine El~Karoui.
\newblock Asymptotic behavior of unregularized and ridge-regularized
  high-dimensional robust regression estimators: rigorous results.
\newblock \emph{arXiv preprint arXiv:1311.2445}, 2013.

\bibitem[El~Karoui et~al.(2013)El~Karoui, Bean, Bickel, Lim, and
  Yu]{el2013robust}
Noureddine El~Karoui, Derek Bean, Peter~J Bickel, Chinghway Lim, and Bin Yu.
\newblock On robust regression with high-dimensional predictors.
\newblock \emph{Proceedings of the National Academy of Sciences}, 110\penalty0
  (36):\penalty0 14557--14562, 2013.

\bibitem[Friedman et~al.(2010)Friedman, Hastie, and
  Tibshirani]{friedman2010regularization}
Jerome Friedman, Trevor Hastie, and Rob Tibshirani.
\newblock Regularization paths for generalized linear models via coordinate
  descent.
\newblock \emph{Journal of statistical software}, 33\penalty0 (1):\penalty0 1,
  2010.

\bibitem[Geman and Geman(1984)]{geman1984stochastic}
Stuart Geman and Donald Geman.
\newblock {Stochastic relaxation, Gibbs distributions, and the Bayesian
  restoration of images}.
\newblock \emph{Pattern Analysis and Machine Intelligence, IEEE Transactions
  on}, \penalty0 (6):\penalty0 721--741, 1984.

\bibitem[George and McCulloch(1993)]{george1993variable}
Edward~I George and Robert~E McCulloch.
\newblock Variable selection via gibbs sampling.
\newblock \emph{Journal of the American Statistical Association}, 88\penalty0
  (423):\penalty0 881--889, 1993.

\bibitem[George and McCulloch(1997)]{george1997approaches}
Edward~I George and Robert~E McCulloch.
\newblock Approaches for bayesian variable selection.
\newblock \emph{Statistica sinica}, pages 339--373, 1997.

\bibitem[Huber(1973)]{huber1973robust}
Peter~J Huber.
\newblock {Robust regression: asymptotics, conjectures and Monte Carlo}.
\newblock \emph{Annals of Statistics}, pages 799--821, 1973.

\bibitem[Huber(2011)]{huber2011robust}
Peter~J Huber.
\newblock \emph{Robust statistics}.
\newblock Springer, 2011.

\bibitem[James and Stein(1961)]{james1961estimation}
William James and Charles Stein.
\newblock Estimation with quadratic loss.
\newblock In \emph{Proceedings of the fourth Berkeley symposium on mathematical
  statistics and probability}, volume~1, pages 361--379, 1961.

\bibitem[Maronna and Yohai(1981)]{maronna1981asymptotic}
Ricardo~A Maronna and Victor~J Yohai.
\newblock {Asymptotic behavior of general M-estimates for regression and scale
  with random carriers}.
\newblock \emph{{Zeitschrift f{\"u}r Wahrscheinlichkeitstheorie und verwandte
  Gebiete}}, 58\penalty0 (1):\penalty0 7--20, 1981.

\bibitem[Park and Casella(2008)]{park2008bayesian}
Trevor Park and George Casella.
\newblock The bayesian lasso.
\newblock \emph{Journal of the American Statistical Association}, 103\penalty0
  (482):\penalty0 681--686, 2008.

\bibitem[Portnoy(1984)]{portnoy1984asymptotic}
Stephen Portnoy.
\newblock {Asymptotic behavior of M-estimators of $p$ regression parameters
  when $p^2/n$ is large. I. Consistency}.
\newblock \emph{Annals of Statistics}, pages 1298--1309, 1984.

\bibitem[Portnoy(1985)]{portnoy1985asymptotic}
Stephen Portnoy.
\newblock {Asymptotic behavior of M estimators of $p$ regression parameters
  when $p^2/n$ is large; II. Normal approximation}.
\newblock \emph{Annals of Statistics}, pages 1403--1417, 1985.

\bibitem[Tibshirani(1996)]{tibshirani1996regression}
Robert Tibshirani.
\newblock Regression shrinkage and selection via the lasso.
\newblock \emph{Journal of the Royal Statistical Society. Series B
  (Methodological)}, pages 267--288, 1996.

\bibitem[van~der Pas et~al.(2014)van~der Pas, Kleijn, and van~der
  Vaart]{van2014horseshoe}
SL~van~der Pas, BJK Kleijn, and AW~van~der Vaart.
\newblock The horseshoe estimator: Posterior concentration around nearly black
  vectors.
\newblock \emph{Electronic Journal of Statistics}, 8\penalty0 (2):\penalty0
  2585--2618, 2014.

\bibitem[West(1987)]{west1987scale}
Mike West.
\newblock On scale mixtures of normal distributions.
\newblock \emph{Biometrika}, 74\penalty0 (3):\penalty0 646--648, 1987.

\bibitem[Yi and Huang(2015)]{yi2015semismooth}
Congrui Yi and Jian Huang.
\newblock Semismooth newton coordinate descent algorithm for elastic-net
  penalized huber loss and quantile regression.
\newblock \emph{arXiv preprint arXiv:1509.02957}, 2015.

\end{thebibliography}
\bibliographystyle{plainnat}

\end{document}